\documentclass[a4paper,10pt]{article}
\usepackage[reqno]{amsmath}
\makeatletter
\newcommand{\leqnomode}{\tagsleft@true}
\newcommand{\reqnomode}{\tagsleft@false}
\makeatother

\usepackage[english,british]{babel}
\usepackage{amsmath,amsthm,amssymb,dsfont,esint}
\usepackage{adjustbox,lipsum}
\usepackage{tikz}
\usepackage{standalone}
\usepackage{mathtools}
\usepackage{fancyhdr}
\usepackage{graphicx}
\usepackage{bibentry}
\usepackage{enumitem}
\usepackage{xcolor}
\usepackage{stackengine}
\usepackage[%pdftex,  %pdflatex,pdftex??????????????????????
            pdfstartview=FitH,
            CJKbookmarks=true,
            bookmarksnumbered=true,
            bookmarksopen=true,
            colorlinks,
            pdfborder=001,
            linkcolor=red,
            citecolor=green,
            ]{hyperref}
\usepackage[a4paper,width=160mm,top=30mm,bottom=30mm]{geometry}
\hypersetup{
    colorlinks,
    citecolor=green,
    filecolor=black,
    linkcolor=red,
    urlcolor=black
}

\newtheorem{theorem}{Theorem}[section]
\newtheorem{lemma}[theorem]{Lemma}
\newtheorem{definition}[theorem]{Definiton}

\newtheorem{corollary}[theorem]{Corollary}

\theoremstyle{definition}

\makeatletter
\newsavebox\myboxA
\newsavebox\myboxB
\newlength\mylenA

\newcommand*\yoverline[2][0.75]{%
    \sbox{\myboxA}{$\m@th#2$}%
    \setbox\myboxB\null% Phantom box
    \ht\myboxB=\ht\myboxA%
    \dp\myboxB=\dp\myboxA%
    \wd\myboxB=#1\wd\myboxA% Scale phantom
    \sbox\myboxB{$\m@th\overline{\copy\myboxB}$}%  Overlined phantom
    \setlength\mylenA{\the\wd\myboxA}%   calc width diff
    \addtolength\mylenA{-\the\wd\myboxB}%
    \ifdim\wd\myboxB<\wd\myboxA%
       \rlap{\hskip 0.5\mylenA\usebox\myboxB}{\usebox\myboxA}%
    \else
        \hskip -0.5\mylenA\rlap{\usebox\myboxA}{\hskip 0.5\mylenA\usebox\myboxB}%
    \fi}
\makeatother

\numberwithin{equation}{section}

\usepackage{mathtools}

\mathtoolsset{showonlyrefs}

\begin{document}

%-------------------------------------

%New Commands

%-------------------------------------

\newcommand{\diver}{\operatorname{div}}
\newcommand{\lin}{\operatorname{Lin}}
\newcommand{\curl}{\operatorname{curl}}
\newcommand{\ran}{\operatorname{Ran}}
\newcommand{\kernel}{\operatorname{Ker}}
\newcommand{\la}{\langle}
\newcommand{\ra}{\rangle}
\newcommand{\N}{\mathbb{N}}
\newcommand{\R}{\mathbb{R}}
\newcommand{\C}{\mathbb{C}}

%%%%%%%%%%%%%%%%%%%%%%%%%%%%%%%%%%%%%%%%
\newcommand{\ld}{\lambda}
\newcommand{\fai}{\varphi}
\newcommand{\0}{0}
\newcommand{\n}{\mathbf{n}}
\newcommand{\uu}{{\boldsymbol{\mathrm{u}}}}
\newcommand{\UU}{{\boldsymbol{\mathrm{U}}}}
\newcommand{\buu}{\bar{{\boldsymbol{\mathrm{u}}}}}
\newcommand{\ten}{\\[4pt]}
\newcommand{\six}{\\[4pt]}
\newcommand{\nb}{\nonumber}
\newcommand{\hgamma}{H_{\Gamma}^1(\OO)}
\newcommand{\opert}{O_{\varepsilon,h}}
\newcommand{\barx}{\bar{x}}
\newcommand{\barf}{\bar{f}}
\newcommand{\hatf}{\hat{f}}
\newcommand{\xoneeps}{x_1^{\varepsilon}}
\newcommand{\xh}{x_h}
\newcommand{\scaled}{\nabla_{1,h}}
\newcommand{\scaledb}{\widehat{\nabla}_{1,\gamma}}
\newcommand{\vare}{\varepsilon}
\newcommand{\A}{{\bf{A}}}
\newcommand{\RR}{{\bf{R}}}
\newcommand{\B}{{\bf{B}}}
\newcommand{\CC}{{\bf{C}}}
\newcommand{\D}{{\bf{D}}}
\newcommand{\K}{{\bf{K}}}
\newcommand{\oo}{{\bf{o}}}
\newcommand{\id}{{\bf{Id}}}
\newcommand{\E}{\mathcal{E}}
\newcommand{\ii}{\mathcal{I}}
\newcommand{\sym}{\mathrm{sym}}
\newcommand{\lt}{\left}
\newcommand{\rt}{\right}
\newcommand{\ro}{{\bf{r}}}
\newcommand{\so}{{\bf{s}}}
\newcommand{\e}{{\bf{e}}}
\newcommand{\ww}{{\boldsymbol{\mathrm{w}}}}
\newcommand{\vv}{{\boldsymbol{\mathrm{v}}}}
\newcommand{\zz}{{\boldsymbol{\mathrm{z}}}}
\newcommand{\U}{{\boldsymbol{\mathrm{U}}}}
\newcommand{\G}{{\boldsymbol{\mathrm{G}}}}
\newcommand{\VV}{{\boldsymbol{\mathrm{V}}}}
\newcommand{\T}{{\boldsymbol{\mathrm{U}}}}
\newcommand{\II}{{\boldsymbol{\mathrm{I}}}}
\newcommand{\ZZ}{{\boldsymbol{\mathrm{Z}}}}
\newcommand{\hKK}{{{\bf{K}}}}
\newcommand{\f}{{\bf{f}}}
\newcommand{\g}{{\bf{g}}}
\newcommand{\lkk}{{\bf{k}}}
\newcommand{\tkk}{{\tilde{\bf{k}}}}
\newcommand{\W}{{\boldsymbol{\mathrm{W}}}}
\newcommand{\Y}{{\boldsymbol{\mathrm{Y}}}}
\newcommand{\EE}{{\boldsymbol{\mathrm{E}}}}
\newcommand{\F}{{\bf{F}}}
\newcommand{\spacev}{\mathcal{V}}
\newcommand{\spacevg}{\mathcal{V}^{\gamma}(\Omega\times S)}
\newcommand{\spacevb}{\bar{\mathcal{V}}^{\gamma}(\Omega\times S)}
\newcommand{\spaces}{\mathcal{S}}
\newcommand{\spacesg}{\mathcal{S}^{\gamma}(\Omega\times S)}
\newcommand{\spacesb}{\bar{\mathcal{S}}^{\gamma}(\Omega\times S)}
\newcommand{\skews}{H^1_{\barx,\mathrm{skew}}}
\newcommand{\kk}{\mathcal{K}}
\newcommand{\OO}{O}
\newcommand{\bhe}{{\bf{B}}_{\vare,h}}
\newcommand{\pp}{{\mathbb{P}}}
\newcommand{\ff}{{\mathcal{F}}}
\newcommand{\mWk}{{\mathcal{W}}^{k,2}(\Omega)}
\newcommand{\mWa}{{\mathcal{W}}^{1,2}(\Omega)}
\newcommand{\mWb}{{\mathcal{W}}^{2,2}(\Omega)}
\newcommand{\twos}{\xrightharpoonup{2}}
\newcommand{\twoss}{\xrightarrow{2}}
\newcommand{\bw}{\bar{w}}
\newcommand{\br}{\bar{{\bf{r}}}}
\newcommand{\bz}{\bar{{\bf{z}}}}
\newcommand{\tw}{{W}}
\newcommand{\tr}{{{\bf{R}}}}
\newcommand{\tz}{{{\bf{Z}}}}
\newcommand{\lo}{{{\bf{o}}}}
\newcommand{\hoo}{H^1_{00}(0,L)}
\newcommand{\ho}{H^1_{0}(0,L)}
\newcommand{\hotwo}{H^1_{0}(0,L;\R^2)}
\newcommand{\hooo}{H^1_{00}(0,L;\R^2)}
\newcommand{\hhooo}{H^1_{00}(0,1;\R^2)}
\newcommand{\dsp}{d_{S}^{\bot}(\barx)}
\newcommand{\LB}{{\bf{\Lambda}}}
\newcommand{\LL}{\mathbb{L}}
\newcommand{\mL}{\mathcal{L}}
\newcommand{\mhL}{\widehat{\mathcal{L}}}
\newcommand{\loc}{\mathrm{loc}}
\newcommand{\tqq}{\mathcal{Q}^{*}}
\newcommand{\tii}{\mathcal{I}^{*}}
\newcommand{\Mts}{\mathbb{M}}
\newcommand{\pot}{\mathrm{pot}}
\newcommand{\tU}{{\widehat{\bf{U}}}}
\newcommand{\tVV}{{\widehat{\bf{V}}}}
\newcommand{\pt}{\partial}
\newcommand{\bg}{\Big}
\newcommand{\hA}{\widehat{{\bf{A}}}}
\newcommand{\hB}{\widehat{{\bf{B}}}}
\newcommand{\hCC}{\widehat{{\bf{C}}}}
\newcommand{\hD}{\widehat{{\bf{D}}}}
\newcommand{\fder}{\partial^{\mathrm{MD}}}
\newcommand{\Var}{\mathrm{Var}}
\newcommand{\pta}{\partial^{0\bot}}
\newcommand{\ptaj}{(\partial^{0\bot})^*}
\newcommand{\ptb}{\partial^{1\bot}}
\newcommand{\ptbj}{(\partial^{1\bot})^*}
\newcommand{\geg}{\Lambda_\vare}
\newcommand{\tpta}{\tilde{\partial}^{0\bot}}
\newcommand{\tptb}{\tilde{\partial}^{1\bot}}
\newcommand{\ua}{u_\alpha}
\newcommand{\pa}{p\alpha}
\newcommand{\qa}{q(1-\alpha)}
\newcommand{\Qa}{Q_\alpha}
\newcommand{\Qb}{Q_\eta}
\newcommand{\ga}{\gamma_\alpha}
\newcommand{\gb}{\gamma_\eta}
\newcommand{\ta}{\theta_\alpha}
\newcommand{\tb}{\theta_\eta}

%%%%%%%%%%%%%%%%%%%%%%%%%%

\newcommand{\mH}{\mathcal{H}}
\newcommand{\mD}{\mathcal{D}}
\newcommand{\csob}{\mathcal{S}}
\newcommand{\mA}{\mathcal{A}}
\newcommand{\mK}{\mathcal{K}}
\newcommand{\mS}{\mathcal{S}}
\newcommand{\mI}{\mathcal{I}}
\newcommand{\tas}{{2_*}}
\newcommand{\tbs}{{2^*}}
\newcommand{\tm}{{\tilde{m}}}
\newcommand{\tdu}{{\phi}}
\newcommand{\tpsi}{{\tilde{\psi}}}
\newcommand{\Z}{{\mathbb{Z}}}
\newcommand{\tsigma}{{\tilde{\sigma}}}
\newcommand{\tg}{{\tilde{g}}}
\newcommand{\tG}{{\tilde{G}}}
\newcommand{\mM}{\mathcal{M}}
\newcommand{\mC}{\mathcal{C}}
\newcommand{\wlim}{{\text{w-lim}}\,}
\newcommand{\SSS}{\mathcal{S}}

\title{Normalized ground states for 3D dipolar Bose-Einstein condensate with  attractive three-body
interactions}
\author{Yongming Luo \thanks{Institut f\"{u}r Wissenschaftliches Rechnen, Technische Universit\"at Dresden, 01069 Dresden, Germany} \ and\
Athanasios Stylianou
\thanks{Institut f\"{u}r Mathematik, Universit\"at Kassel, 34132 Kassel, Germany}}
\date{}
\maketitle

\begin{abstract}
We study the existence of normalized ground states for the 3D dipolar Bose-Einstein condensate
equation with attractive three-body interactions:
\begin{align}\label{abstract eq dipolar}
-\Delta u+\beta u+\ld_1|u|^2 u+\ld_2 (K*|u|^2)u-|u|^4u=0.\tag{DBEC}
\end{align}
When $\ld_2=0$ or $u$ is radial, \eqref{abstract eq dipolar} reduces to the cubic-quintic NLS
\begin{align}\label{abstract eq cqnls}
-\Delta u+\beta u+\ld_1|u|^2 u-|u|^4u=0\tag{CQNLS},
\end{align}
which has been recently studied by Soave in \cite{SoaveCritical}. In particular, it was shown that
for any $\ld_1<0$ and $c>0$, \eqref{abstract eq cqnls} possesses a radially symmetric ground state
solution with mass $c$ and for $\ld_1\geq 0$, \eqref{abstract eq cqnls} has no non-trivial solution.
We show that by adding a dipole-dipole interaction to \eqref{abstract eq cqnls}, the geometric
nature of \eqref{abstract eq cqnls} changes dramatically and techniques as the ones from
\cite{SoaveCritical} cannot be used anymore to obtain similar results. More precisely, due to the
axisymmetric nature of the dipole-dipole interaction potential, the energy corresponding to
\eqref{abstract eq dipolar} is not stable under symmetric rearrangements, hence conventional
arguments based on the radial symmetry of solutions are inapplicable. We will overcome this
difficulty by appealing to subtle variational and perturbative methods and prove the following:
\begin{itemize}
\item[(i)] If the pair $(\ld_1,\ld_2)$ is \textit{unstable} and $\ld_1<0$, then for any $c>0$,
\eqref{abstract eq dipolar} has a ground state solution with mass $c$.
\item[(ii)] If the pair $(\ld_1,\ld_2)$ is unstable and $\ld_1\geq 0$, then there exists some
$c^*=c^*(\ld_1,\ld_2)\geq 0$ such that for all $c>c^*$, \eqref{abstract eq dipolar} has a ground
state solution with mass $c$. Moreover, any non-trivial solution of \eqref{abstract eq dipolar} in
this case must be non-radial.
\item[(iii)] If the pair $(\ld_1,\ld_2)$ is \textit{stable}, then \eqref{abstract eq dipolar} has
no non-trivial solutions.
\end{itemize}
\end{abstract}

% \begin{abstract}
% \end{abstract}

%\footnotetext[1]{\textbf{Keywords:} }
%\footnotetext[2]{\textbf{2020 AMS Subject Classification:} }

\section{Introduction and main results}
In this paper, we prove existence of solitary waves for the equation modeling 3D dipolar
Bose-Einstein condensates (DBEC) with attractive three-body interactions:
\begin{align}\label{nls}
i\partial_t \phi=-\Delta \phi+\lambda_1|\phi|^2\phi+\lambda_2(K*|\phi|^2)\phi-|\phi|^4\phi,
\end{align}
where $\ld_1,\ld_2\in\R$ are given constants and the dipole-dipole kernel $K$ is defined by
$$K(x)=\frac{x_1^2+x_2^2-2x_3^2}{|x|^5}.$$
To be more precise, we will be seeking solitary wave solutions $u$ of \eqref{nls} which satisfy the
stationary DBEC equation:
\begin{align}\label{pde}
-\Delta u+\beta u+\lambda_1 |u|^2 u+\lambda_2 (K*|u|^2)u-|u|^4 u=0
\end{align}
and possess a prescribed mass $\|u\|_2^2=c$ for a given $c\in(0,\infty)$. It then follows directly
that the function $\phi(t,x)=e^{i\beta t}u(x)$ is a solution to \eqref{nls}, for any solution $u$
of \eqref{pde}. Equations \eqref{nls} and \eqref{pde} can be generalized to
\begin{align}\label{gnls}
i\partial_t \phi=-\Delta \phi+\lambda_1|\phi|^2\phi+\lambda_2(K*|\phi|^2)\phi+\ld_3|\phi|^p\phi
\end{align}
and
\begin{align}\label{gpde}
-\Delta u+\beta u+\lambda_1 |u|^2 u+\lambda_2 (K*|u|^2)u+\ld_3|u|^p u=0
\end{align}
with $\ld_3\in\C$ and $p\in(0,\infty)$ respectively\footnote{More generally, an external trapping
potential $V_{\rm ext}$ should also be contained in \eqref{gnls}. We consider in this paper the case
$V_{\rm ext}=0$, which corresponds to the so-called \textit{self-bounded} model.}. The
potentials $|u|^2 u$ and $(K*|u|^2)u$ describe the two-body and long range dipole-dipole
interactions respectively. For the potential $|u|^p u$, the cases $p=5$ and $p=6$ correspond to the
Lee-Huang-Yang-correction (LHY-correction) and three-body interaction respectively. When in a
physical experiment the parameters $\ld_1,\ld_2$ are tuned so that they lie in the
\textit{unstable regime} (see Definition \ref{def stable}), the classical Gross-Pitaevskii
theory would predict a collapse of the gas, which was not seen during the experiment. In order
to stabilize the Gross-Pitaevskii equation, it has been suggested to incorporate a
higher order repulsive\footnote{In the case where the three-body
interaction is under consideration, the parameter $\ld_3$ might also have non-trivial imaginary
part, indicating a three-body loss effect.} term such as the LHY-correction or three-body
interactions; we refer to the papers
\cite{Blakie2016,Ferrier-BarbutEtAl2016,KadauEtAl2016,LahayeEtAl2009,SchmittEtAl2016,
ZhouLiangZhang2010} and the references therein for a more comprehensive introduction on the physical
background of \eqref{gnls}. There has also been an ongoing investigation of models
for collapsing Bose-Einstein condensates with attractive three-body interaction (see for instance
\cite{Collapse2011,Koehler2002,Sabari2010,Singh2016}). The subject of this paper is to study this
aspect of the DBEC, which itself presents a mathematically challenging problem.

The first mathematically rigorous analysis for \eqref{gnls} dates back to the work of Carles,
Markowich and Sparber in \cite{Carles2008}, where \eqref{gnls} was considered without any higher
order term ($\ld_3=0$). The authors proved local and global well-posedness and finite time blow-up
results for \eqref{gnls}. Particularly, 1D and 2D DBEC-models were also derived from the 3D model
via dimension reduction. Later, Antonelli and Sparber \cite{AntonelliSparber2011} proved existence
of ground states for \eqref{gpde} (again in the case $\ld_3=0$) in the unstable regime using the
so-called Weinstein functional method. Further regularity and symmetry results of the ground states
were also established. The existence of normalized ground states for \eqref{gnls} in the
unstable regime without higher order term was later proved by Bellazzini and Jeanjean
\cite{BellazziniJeanjean2016} using mountain pass arguments. In particular, it was shown that
ground states obtained by mountain pass are automatically the Weinstein optimizers
found in \cite{AntonelliSparber2011}. Further existence, stability and well-posedness results
for \eqref{gnls} and \eqref{gpde} with or without an external trapping potential were also obtained
in \cite{BellazziniJeanjean2016}. In \cite{BellazziniForcella2019}, Bellazzini and Forcella were
able to utilize the ground states given in \cite{AntonelliSparber2011} and
\cite{BellazziniJeanjean2016} to formulate a sharp scattering threshold for \eqref{gnls} without
higher order term. The first results for \eqref{gnls} and \eqref{gpde} with higher order term were
given by the Authors \cite{LuoStylianouJMAA,LuoStylianouDCDSB}, where the cases $\ld_3<0,p=5$ and
$\ld_3>0,p\in(4,6]$ were studied. We also refer to
\cite{ArdilaDipolar,BaoBenCai2012,BaoCai2018,BaoCaiWang2010,BellazziniForcellaCVPDE,
CarlesHajeiej2015,DinhDipolar1,
DinhDipolar2,EychenneRougerie2019,FengCaoLiu2021,LiMa2021,TriayDipolar2018} and the references
therein for recent analytical and numerical progress on \eqref{gnls} and \eqref{gpde}.

From now on we focus on the DBEC \eqref{pde} with attractive three-body interactions. A solitary
wave solution of \eqref{pde} is of fundamental importance for studying \eqref{nls}, since it might
be the only observable quantity in physical experiments and can be seen as a balance point between
linear and nonlinear effects. Here, we will be looking for ground state solutions with prescribed
mass, i.e., the total number of particles in the gas.

Before we state the main results of the paper, we firstly fix some definitions and notation. The Hamiltonian $E(u)$ corresponding to \eqref{pde} is defined by
\begin{align*}
E(u)=\frac{1}{2}\|\nabla u\|^2_{2}+\frac{1}{4}\lambda_1\|u\|^4_4+\lambda_2\int_{\R^3}(K*|u|^2)|u|^2\,dx-\frac{1}{6}\|u\|_6^6.
\end{align*}
For $c>0$, the manifold $S(c)$ is defined by
\begin{align*}
S(c)=\{u\in H^1(\R^3):\|u\|^2_{2}=c\}.
\end{align*}
The definition of \textit{unstable} and \textit{stable} regimes is given as follows:
\begin{definition}[Unstable and stable regimes]\label{def stable}
We define the unstable and stable regimes as follows:
\begin{itemize}
\item[(i)] The pair $(\lambda_1,\lambda_2)$ is said to be in the \textbf{unstable} regime if
$$\lambda_2>0\ \text{ and }\ \lambda_1-\frac{4\pi}{3}\lambda_2< 0$$
or
$$\lambda_2< 0\ \text{ and }\ \lambda_1+\frac{8\pi}{3}\lambda_2< 0;$$

%%%%%%%%%%%

\item[(ii)] The pair $(\lambda_1,\lambda_2)$ is said to be in the \textbf{stable} regime if
$$\lambda_2>0\ \text{ and }\ \lambda_1-\frac{4\pi}{3}\lambda_2>0$$
or
$$\lambda_2< 0\ \text{ and }\ \lambda_1+\frac{8\pi}{3}\lambda_2>0;$$
\end{itemize}
\end{definition}
We use the following form of the Fourier transform:
$$ \hat{f}(\xi)=\int_{\R^3} f(x)e^{-i\xi\cdot x}\,dx.$$
Due to \cite{Carles2008}, the Fourier transform $\widehat K$ of $K$ is given by
$$
\widehat{K}(\xi)=\frac{4\pi}{3}\frac{2\xi_3^2-\xi_1^2-\xi_2^2}{|\xi|^2}\in
\left[-\frac{4\pi}{3},\frac{8\pi}{3}\right].$$
The following quantities will also be used throughout the paper:
\begin{align*}
A(u)&:=\|\nabla u\|^2_2,\\
B(u)&:=\lambda_1\|u\|_4^4+\lambda_2\int_{\R^3} (K*|u|^2)|u|^2\,dx,\\
C(u)&:=\|u\|_6^6,\\
Q(u)&:=A(u)+\frac{3}{4}B(u)-C(u).
\end{align*}
Next, we define the set $V(c)$ by
\begin{align}\label{def of Vc}
V(c)=\{u\in S(c):Q(u)=0\}
\end{align}
and we define the variational problem $\gamma(c)$ by
\begin{align}\label{variational gamma c}
\gamma(c):=\inf\{E(u):u\in V(c)\}.
\end{align}
Finally, for $u\in H^1(\R^3)$ and $t>0$ we define the function $u^t$ by
\begin{align}\label{def l2 scaling}
u^t(x)=t^{\frac{3}{2}}u(tx).
\end{align}
It is a direct calculation to check that the rescaling \eqref{def l2 scaling} leaves the
$L^2$-norm invariant.

The main result of the present paper is the following:
\begin{theorem}[Existence of ground states]\label{ground state}
The following statements hold true:
\begin{itemize}
\item[(i)] If $(\lambda_1,\lambda_2)$ is an unstable pair and $\ld_1<0$, then for any $c\in(0,\infty)$ the variational problem \eqref{variational gamma c} has a positive optimizer $u_c\in S(c)$.

\item[(ii)]If $(\lambda_1,\lambda_2)$ is an unstable pair and $\ld_1\geq 0$, then there exists some $c^*=c^*(\ld_1,\ld_2)\in[0,\infty)$ such that $c\mapsto\gamma(c)$ is constantly equal to $\frac{\SSS^{\frac32}}{3}$ on $(0,c^*]$ and strictly less than $\frac{\SSS^{\frac32}}{3}$ on $(c^*,\infty)$. Furthermore, for any $c\in(c^*,\infty)$ the variational problem \eqref{variational gamma c} has a positive optimizer $u_c\in S(c)$.
\end{itemize}
Moreover, the optimizers given by (i) and (ii) are solutions of \eqref{pde} with $\beta=\beta_c>0$.

\begin{itemize}
\item[(iii)] Let $u_c$ be the solution of \eqref{pde} given by (ii). Then for any $x\in\R^3$,
$u_c(\cdot+x)$ is not radially symmetric.
\item[(iv)] If $(\lambda_1,\lambda_2)$ is a stable pair, then \eqref{pde} has no non-trivial solution in $H^1(\R^3)$.
\end{itemize}
\end{theorem}

We make a couple of comments on Theorem \ref{ground state}: the quintic potential $|u|^4 u$ is
energy-critical in 3D. Moreover, since $K$ has vanishing sphere integral, it is in fact a
Calderon-Zygmund kernel and therefore a bounded mapping from $L^p$ to $L^p$ for all
$p\in(1,\infty)$. Hence \eqref{pde} can be seen as the focusing energy-critical NLS perturbed by a
cubic-like lower order term. The similar cubic-quintic model
\begin{align}\label{cqnls}
-\Delta u+\beta u+\ld_1|u|^2 u-|u|^4u=0
\end{align}
has been recently studied by Soave in \cite{SoaveCritical} (in fact, a general class of focusing
energy-critical NLS with combined powers including \eqref{cqnls} was studied therein). In
particular, Soave proved that for any $\ld_1<0$ and $c>0$, $\gamma(c)$ defined by
\eqref{variational gamma c} (corresponding to \eqref{cqnls}) has a normalized positive, radially
symmetric ground state $u_c\in S_c$ and $\gamma(c)\in \big(0,\frac{\SSS^{\frac32}}{3}\big)$,
where $\SSS$ is the best constant for the Sobolev inequality:
\begin{align*}
\SSS=\inf_{u\in\mathcal{D}^{1,2}(\R^3)}\frac{\|\nabla u\|_2^2}{\|u\|_6^2}.
\end{align*}
On the contrary, for $\ld_1\geq 0$, Soave showed that \eqref{cqnls} has no non-trivial
solution\footnote{This was in fact originally shown in the case $\ld_1>0$, but the proof extends
verbatim to $\ld_1\geq 0$.}.

It remains an open problem whether there exist normalized ground states of \eqref{pde} for
$c\in(0,c^*)$ in the case where $(\ld_1,\ld_2)$ is unstable and $\ld_1\geq 0$. As it will become
clear from the proof of Theorem \ref{ground state}, the existence of ground states is, loosely
speaking, equivalent to showing that $\gamma(c)$ is strictly smaller than
$\frac{\SSS^{\frac32}}{3}$. In the case of \eqref{cqnls}, Soave showed this directly by constructing
a sequence of test functions that asymptotically meets the assumption, where the construction of the
test functions is based on certain delicate cut-off refinement of the radial Aubin-Talenti ground
states. However, the radial symmetry is principally incompatible with the dipolar potential and
Soave's arguments can not be utilized for finding a reasonable minimizing sequence in the case
$\ld_1\geq 0$. In view of such consideration, we conjecture that in the case of Theorem \ref{ground
state} (ii), the number $c^*$ is strictly positive and there is no ground states of $\gamma(c)$ for
all $c\in(0,c^*]$. We note, however, that this is not contradicting Theorem \ref{ground state} (ii),
since for large mass we can find a minimizing sequence without invoking the radial symmetry at all,
while simultaneously the gradient and quintic potential energies remain controllable.

The rest of the paper is organized as follows: In Section \ref{sec:prelim} we provide some auxiliary lemmas which will be useful for proving the main results. In Section \ref{sec:proof main theorem} we prove our main results.

\section{Preliminaries}\label{sec:prelim}
In this section we collect some useful auxiliary results which will be later used in the proof of Theorem \ref{ground state}.
\subsection{Pohozaev identity}
We begin with the well-known Pohozaev identity.
\begin{lemma}\label{lemma pohozaev}
Let $u$ be a solution of \eqref{pde}. Then
\begin{align}
A(u)+3\beta\|u\|_2^2+\frac{3}{2}B(u)-C(u)=0.\label{pohozaev}
\end{align}
Consequently, we have
\begin{align}
4\beta\|u\|_2^2&=-B(u),\label{beta bu}\\
Q(u)&=0\label{pohozaev3}.
\end{align}
\end{lemma}
\begin{proof}
\eqref{pohozaev} follows generally by multiplying \eqref{pde} with $x\cdot\nabla \bar u$ and then
integrating by parts (for more details see \cite{LuoStylianouJMAA}). On the other hand, by
multiplying \eqref{pde} with $\bar{u}$ and integrating by parts we obtain
\begin{align}\label{pohozaev2}
A(u)+\beta\|u\|_2^2+B(u)-C(u)=0.
\end{align}
Eliminating $A(u)$ and $C(u)$ in \eqref{pohozaev} and \eqref{pohozaev2} yields \eqref{beta bu}; Eliminating $\beta\|u\|_2^2$ in \eqref{pohozaev} and \eqref{pohozaev2} yields \eqref{pohozaev3}.
\end{proof}

\subsection{Characterization of optimizers of $\gamma(c)$}
In the following, we show that any optimizer of $\gamma_c$ is automatically a solution of \eqref{pde}.
\begin{lemma}\label{vc implies sc}
Let $c>0$. If $\gamma(c)$ is attained at some $u\in V(c)$, then $u$ solves \eqref{pde}.
\end{lemma}
\begin{proof}
Suppose that $\gamma(c)$ is attained at $u \in V(c)$. Then in view of the Lagrange multiplier theorem, there exist $\mu_1,\mu_2\in\C$ such that
\begin{align*}
E'(u)-\mu_1Q'(u)-2\mu_2 u=0,
\end{align*}
or equivalently
\begin{align}\label{pde on Vc}
(1-2\mu_1)(-\Delta u )+(1-3\mu_1)(\lambda_1|u |^2+\lambda_2(K*|u |^2))u +(6\mu_1-1)|u |^4u -2\mu_2 u =0.
\end{align}
The Pohozaev identity corresponding to \eqref{pde on Vc} is given by
\begin{align}
0=(1-2\mu_1)A(u )+\frac{3}{4}(1-3\mu_1)B(u )+(6\mu_1-1)C(u).\label{pde on Vc 2}
\end{align}
Eliminating $B(u)$ in \eqref{pde on Vc} and \eqref{pde on Vc 2} and using the fact that $Q(u)=0$, we infer that
\begin{align*}
\mu_1(A(u )+3C(u ))=0.
\end{align*}
Since $A(u )+3C(u )>0$, we know that $\mu_1 =0$ and the proof is complete.
\end{proof}

\subsection{The energy landscape along the $L^2$-invariant scaling}
For any function $u\in H^1(\R^3)$, we recall that $u^t$ is the $L^2$-invariant scaling of $u$
defined by \eqref{def l2 scaling}. The following lemma shows that we can always find some $t^*>0$
such that $u^{t^*}\in V(c)$ is a local maximum of the mapping $t\mapsto E(u^t)$. Thus $u^{t^*}$ can
be viewed as the peak of a mountain pass.
\begin{lemma}\label{monotoneproperty}
Let $c>0$ and $u\in S(c)$. Then:
\begin{enumerate}
\item[(i)] $\displaystyle\frac{\partial}{\partial t}E(u^t)=\frac{Q(u^t)}{t}$, for all $t>0$.
\item[(ii)] There exists a $t^*>0$ such that $u^{t^*}\in V(c)$.
\item[(iii)] We have $t^*(u)<1$ if and only if $Q(u)<0$. Moreover, $t^*(u)=1$ if and only if $Q(u)=0$.
\item[(iv)] The following inequalities hold:
\begin{equation*}
Q(u^t) \left\{
\begin{array}{lr}
             >0, &t\in(0,t^*(u)) ,\\
             <0, &t\in(t^*(u),\infty).
             \end{array}
\right.
\end{equation*}
\item[(v)] $E(u^t)<E(u^{t^*})$ for all $t>0$ with $t\neq t^*$.
\end{enumerate}
\end{lemma}
\begin{proof}
(i) follows from direct calculation. Next, define $y(t):= \frac{\partial}{\partial t}E(u^t)$. Then
\begin{align*}
y(t)&=tA(u)+\frac{3}{4}t^2B(u)-t^5C(u),\\
y'(t)&=A(u)+\frac{3t}{2}B(u)-5t^4C(u),\\
y''(t)&=\frac{3}{2}B(u)-20t^3C(u).
\end{align*}
If $B(u)\leq 0$, then $y''(t)$ is negative on $(0,\infty)$; If $B(u) >0$, then $y''(t)$ is positive on $(0,(\frac{3B}{40C})^{\frac{1}{3}})$ and negative on $(-(\frac{3B}{40C})^{\frac{1}{3}},\infty)$. Since $y'(0)=A(u)>0$ and $y'(t)\to -\infty$ as $t\to \infty$, we conclude simultaneously from both cases that there exists a $t_0>0$ such that $y'(t)$ is positive on $(0,t_0)$ and negative on $(t_0,\infty)$. From the expression for $y(t)$ we obtain that $\lim_{t\searrow 0^+}y(t)=0$ and $\displaystyle \lim_{t\to\infty}y(t)=-\infty$. Thus $y(t)$ has a zero at $t^*>t_0$, $y(t)$ is positive on $(0,t^*)$ and $y(t)$ is negative on $(t^*,\infty)$. Since $y(t)=\frac{\partial E(u^t)}{\partial t}=\frac{Q(u^t)}{t}$, (ii) and (iv) are shown. For (iii), we first let $Q(u)<0$. Then
$$0>Q(u)=\frac{Q(u^1)}{1}=y(1),$$
which is only possible if $t^*<1$. Conversely, let $t^*<1$. Then
$$Q(u)=y(1)<y(t^*)<0. $$
This completes the proof of (iii). To see (v), we use that
$$ E(u^{t^*})=E(u^t)+\int_t^{t^*}y(s)\,ds.$$
Then (v) follows from the fact that $y(t)$ is positive on $(0,t^*)$ and negative on $(t^*,\infty)$.
\end{proof}

\subsection{Palais-Smale sequences with vanishing virial}
In this subsection we prove the existence of a bounded Palais-Smale (PS) sequence. Since $E(u)$ is
unbounded below on $S(c)$ (which can be easily verified using H\"older's inequality), the
boundedness of a PS-sequence does not directly follow from the mountain pass geometry. Nevertheless,
by Lemma \ref{vc implies sc} we will be seeking optimizers on the manifold $V(c)$ containing
functions $u$ with vanishing virial $Q(u)$, from which the boundedness of the PS-sequence follows.
Hence, the problem reduces to finding a PS-sequence with vanishing virial. To show this, we firstly
introduce the following definition of \textit{homotopy-stable family}:
\begin{definition}[Homotopy-stable family, \protect{\cite[Def. 3.1]{Ghoussoub1993}}]
Let $B$ be a closed subset of a metric space $X$. We say that a class $\mathcal{F}$ of compact subsets
of $X$ is a homotopy-stable family with closed boundary $B$ if
\begin{itemize}
\item[(i)] $B$ is contained in every set in $\mathcal{F}$ contains $B$;
\item[(ii)] For any $A \in\mathcal{F}$ and any $\eta \in C([0,1] \times X,X )$ satisfying $\eta(t,x) = x$ for all $(t,x) \in(\{0\}\times X ) \cup([0,1] \times B)$, we have $\eta(\{1\} \times A) \in \mathcal{F}$.
\end{itemize}
If $B$ is empty, we call $\mathcal{F}$ a homotopy-stable family without boundary.
\end{definition}
\begin{lemma}[Existence of PS-sequence with vanishing virial]\label{ps seq}
For each $c>0$ there exists s PS-sequence $(u_n)_n\subset S(c)$ with vanishing virial, in other words $(u_n)_n$ satisfies
\begin{align*}
E(u_n)&=\gamma(c)+o(1),\\
E'(u_n)&=o(1),\\
\mathrm{dist}(u_n,V(c))&=o(1)
\end{align*}
with $o(1)=o_n(1)$ as $n\to \infty$.
\end{lemma}
\begin{proof}
We define
\begin{align*}
X&:=S(c),\\
\mathcal{F}&:=\{\{u\}:u\in S(c)\},\\
B&:=\varnothing.
\end{align*}
It follows directly that $\mathcal{F}$ is a homotopy-stable family without boundary. Next, we define
\begin{align*}
\varphi(u)=E(u^{t^*}).
\end{align*}
Then
\begin{align*}
\inf_{A\in\mathcal{F}}\max_{u\in A}\varphi(u)=\inf_{u\in S(c)}E(u^{t^*})=:\gamma_2(c).
\end{align*}
We show that $\gamma(c)=\gamma_2(c)$. Since $u^{t^*}\in V(c)$, it follows that $\gamma_2(c)\geq \gamma(c)$. On the other hand, if $u\in V(c)$, then
$$E(u)\geq \inf_{v\in V(c)}E(v)=\inf_{v\in V(c)}E(v^{t^*})\geq \inf_{v\in S(c)}E(v^{t^*})=\gamma_2(c),$$
which implies that $\gamma(c)\geq \gamma_2(c)$. It is also standard to check that $S(c)$ is a $C^1$-manifold and $\varphi$ is a $C^1$ functional on $S(c)$. Thus the claim follows from \cite[Thm. 3.2]{Ghoussoub1993} by setting $A_n=\{u_n\}\in\mathcal{F}$ therein, where $(u_n)_n\subset V(c)$ is a minimizing sequence , i.e. $E(u_n)=\gamma(c)+o(1)$.
\end{proof}

The following corollary is an immediate consequence of Lemma \ref{ps seq}. The proof is standard and we refer to \cite[Prop. 4.1]{Bellazzini2013} for related arguments.
\begin{corollary}\label{ps cor}
Let $c>0$ and $(u_n)_n\subset S(c)$ be the bounded Palais-Smale sequence constructed in Lemma \ref{ps seq}. Then there exist $u\in H^1(\R^3,\C)$, $\beta\in\mathbb R$, a (not relabeled) subsequence $(u_n)_n$ and a sequence $(\beta_n)_n\subset\R$ such that:
\begin{enumerate}
\item $u_n\rightharpoonup u$ in $H^1(\R^3)$.
\item $\displaystyle \beta_n\to\beta$ in $\R$.
\item $-\Delta u_n+\beta_nu_n +\lambda_1|u_n|^2\,u_n+\lambda_2 (K*|u_n|^2)u_n-|u_n|^4u_n\to 0$ in $H^{-1}(\R^3)$.
\item $-\Delta u+\beta u +\lambda_1|u|^2u+\lambda_2 (K*|u|^2)u-|u|^4u= 0$ in $H^{-1}(\R^3)$.
\end{enumerate}
\end{corollary}

\subsection{Nonvanishing weak limit}
We recall that from Corollary \ref{ps cor} we have found a function $u$ that will solve \eqref{pde}
for some $\beta\in \R$. However, it is \textit{a priori} unclear whether $u$ is non-vanishing. We
show that this is indeed the case when $\gamma(c)$ is strictly smaller than
$\frac{\SSS^{\frac32}}{3}$. The original proof of Soave \cite{SoaveCritical} relied on radial
symmetry arguments\footnote{In the original proof of Soave, the condition $\|u_n\|_4= o(1)$ was in
fact a consequence of the Strauss' compactness lemma for radial functions.} and is not applicable in
our case. We will remove this restriction using the $pqr$-lemma from \cite[Lem. 2.1]{FLL1986}.
\begin{lemma}\label{u not equal zero}
Let $c>0$. Suppose that the mountain pass level $\gamma(c)$ of the PS-sequence $(u_n)$ given by Lemma \ref{ps seq} satisfies
$$ \gamma(c)<\frac{\SSS^{\frac{3}{2}}}{3}\qquad\text{and}\qquad \gamma(c)\neq 0.$$
Then $(u_n)_n$ has a nonzero weak limit $u$ in $H^1(\R^3)$.
\end{lemma}
\begin{proof}
By Corollary \ref{ps cor} we know that $(u_n)_n$ is a bounded sequence in $H^1(\R^3)$, hence also in
$L^p(\R^3)$ for all $p\in[2,6]$. We can then distinguish between two cases:
\begin{itemize}
\item[(i)]$\|u_n\|_4= o(1)$. Then the claim follows from \cite[Lem. 3.3]{SoaveCritical}.

\item[(ii)]$\|u_n\|_4\neq o(1)$. Then the claim follows from the $pqr$-lemma (\cite[Lem.
2.1]{FLL1986}) and the Lieb-translation (\cite[Lem. 6]{Lieb1983}).
\end{itemize}
\end{proof}

\subsection{A qualitative description of the mapping $c\mapsto\gamma(c)$}
We first show that the number $\gamma(c)$ is always positive and will never exceed
$\frac{\SSS^{\frac{3}{2}}}{3}$.
\begin{lemma}\label{range of gamma c}
For all $c>0$ we have $\gamma(c)\in(0,\frac{\SSS^{\frac{3}{2}}}{3}]$.
\end{lemma}
\begin{proof}
Let $u\in V(c)$. Using the Gagliardo-Nirenberg and Sobolev inequalities we obtain that
\begin{align*}
\|\nabla u\|_2^2=A(u)=-\frac{3}{4}B(u)+C(u)\leq C(c^{\frac{1}{2}}\|\nabla u\|^3_2+\|\nabla u\|^6_2).
\end{align*}
Since $c\neq 0$ we can divide by $\|\nabla u\|_2^2$ and modify the constants to obtain that
$$ \inf_{u\in V(c)}\|\nabla u\|_2>0.$$
Since
$$E(u)=E(u)-\frac{1}{3}Q(u)=\frac{1}{6}(A(u)+C(u))\geq \frac{1}{6}A(u)=\frac{1}{6}\|\nabla u\|_2^2$$
for $u\in V(c)$, the lower bound follows by taking infimum over $V(c)$ on both sides. It is left to show the upper bound. We define
\begin{align*}
\tilde{E}(u)&:=\frac{1}{2}\|\nabla u\|_2^2+\frac{1}{4}\lambda_1\|u\|_4^4-\frac{1}{6}\|u\|_6^6,\\
\tilde{Q}(u)&:=\|\nabla u\|_2^2+\frac{3}{4}\lambda_1\|u\|_4^4-\|u\|_6^6.
\end{align*}
Then $\tilde{E}(u)$ and $\tilde{Q}(u)$ are the energy and virial corresponding to the equation
\begin{align}\label{no dipolar}
-\Delta u+\beta u+\lambda_1 |u|^2u-|u|^4u=0
\end{align}
respectively. For a radially symmetric $u$ we have
\begin{align*}
\tilde{E}(u)&:=E(u),\\
\tilde{Q}(u)&:=Q(u).
\end{align*}
Now we define
\begin{align}
u_{\varepsilon}(x) &:=\varphi(x)\cdot(\frac{\varepsilon}{\varepsilon^2+|x|^2})^{\frac{1}{2}},
\label{eq 1}\\
v_{\varepsilon} &:=c^{\frac{1}{2}}\frac{u_{\varepsilon}}{\|u_{|\varepsilon}\|_2}\label{eq_2},
\end{align}
where $\varphi\in C^\infty_c(\R^3)$ is a real radial cut-off function with $c\equiv 1$ in $B_1$,
$\phi\equiv 0$ in $B_2^c$ and $\phi$ is radially decreasing. Denote by $\tilde{t}^*$ the number
introduced in Lemma \ref{monotoneproperty}, but corresponding to \eqref{no dipolar}. Due to
\cite[Lem. 6.4]{SoaveCritical} there exist $C_1,C_2>0$ such that
\begin{align}\label{threshold}
\tilde{E}(v_{\varepsilon}^{\tilde{t}^*})\leq \frac{\SSS^{\frac{3}{2}}}{3}+C_1\varepsilon^{\frac{1}{2}}+C_2\varepsilon=\frac{\SSS^{\frac{3}{2}}}{3}
+O(\vare^{\frac12})
\end{align}
for all sufficiently small $\varepsilon$. Since $v_{\varepsilon}^{\tilde{t}^*}$ is radially
symmetric, using \eqref{eq 1} and \eqref{eq_2} we obtain that $v_{\varepsilon}^{\tilde{t}^*}\in
V(c)$ and $E(v_{\varepsilon}^{\tilde{t}^*})=\tilde{E}(v_{\varepsilon}^{\tilde{t}^*})$. The upper
bound follows by the definition of $\gamma(c)$.
\end{proof}
\begin{lemma}\label{infV(c)}
The curve $c\mapsto\gamma(c)$ is non-increasing and continuous on $(0,\infty)$.
\end{lemma}
\begin{proof}
The proof is standard, see for instance \cite{Bellazzini2013}.
\end{proof}

The following lemma is fundamental for the proof of Theorem \ref{ground state} (ii).
\begin{lemma}\label{vanishing of gamma c}
Let $(\ld_1,\ld_2)$ be unstable. Then $\gamma(c)\to 0$ as $c\to\infty$.
\end{lemma}
\begin{proof}
By \cite[Remark 4.1]{BellazziniJeanjean2016} we have that the set
$$ \{u\in H^1(\R^3):B(u)< 0\}$$
is not empty. Assuming $B(u)<0$ and using the Fourier transform and the Gagliardo-Nierenberg
inequality, we estimate
\[
 -B(u)\leq \frac{\|u\|_4^4}{(2\pi)^3}\,
\max\left\{\left|\lambda_1-\lambda_2\frac{4\pi}{3}\right|,\left|\lambda_1+\lambda_2\frac{8\pi}{3}
\right|\right\}\leq \frac{\|u\|_2\, \|\nabla u\|_2^3}{(2\pi)^3}\,
\max\left\{\left|\lambda_1-\lambda_2\frac{4\pi}{3}\right|,\left|\lambda_1+\lambda_2\frac{8\pi}{3}
\right|\right\},
\]
so that
\begin{align}\label{gagliardo nirenberg nonlocal}
\mathrm{C}_{\rm GN}:=\inf_{u\in H^1(\R^3)\setminus\{0\},B(u)<0}\frac{\|u\|_2\|\nabla
u\|_2^3}{-B(u)} >0.
\end{align}
Let $(u_n)_n$ be a minimizing sequence for
 \eqref{gagliardo nirenberg nonlocal}. The rescaling $u(x)\to q\,u(s\,x)$ leaves the minimized
quantity in \eqref{gagliardo nirenberg nonlocal} invariant, so we may assume that
\begin{align}
\|u_n\|_2=n,\quad -B(u_n)=1,
\end{align}
and, therefore,
\begin{align}
\|\nabla u_n\|_2=n^{-\frac{1}{3}}(\mathrm{C}_{\rm GN}+o(1))^{-\frac{1}{3}}.
\end{align}
Using Lemma \ref{monotoneproperty} we find $\theta_n\in(0,\infty)$ such that
$Q(u_n^{\theta_n})=0$. This is equivalent to
\begin{align}
0&=\theta^2_n\|\nabla u_n\|^2_2+\frac{3}{4}\theta_n^3B(u_n)-\theta_n^6\|u_n\|_6^6\nonumber\\
&=\theta^2_n n^{-\frac{2}{3}}(\mathrm{C}_{\rm GN}+o(1))^{-\frac{2}{3}}-\frac34
\theta_n^3-\theta_n^6\|u_n\|_6^6,
\end{align}
which implies
\begin{align}\label{crutial eq}
\theta_n^4\|u_n\|_6^6= n^{-\frac{2}{3}}(\mathrm{C}_{\rm GN}+o(1))^{-\frac{2}{3}}-\frac34\theta_n.
\end{align}
If either $\liminf\limits_{n\to\infty}\theta_n>0$ or $\limsup\limits_{n\to\infty}\theta_n>0$, then
for sufficiently large $n$, the right hand side of \eqref{crutial eq} will be negative, while the
left hand side is always positive, which is a contradiction. Thus $\lim\limits_{n\to
\infty}\theta_n=0$. Consequently,
\begin{align}
\gamma_{n^2}&\leq E(u_n^{\theta_n})\nonumber\\
&=\frac{1}{2}\theta^2_n n^{-\frac{2}{3}}(\mathrm{C}_{\rm GN}+o(1))^{-\frac{2}{3}}-\frac14
\theta_n^3-\frac{1}{6}\theta_n^6\|u_n\|_6^6\nonumber\\
&=\frac{1}{2}\theta^2_n n^{-\frac{2}{3}}(\mathrm{C}_{\rm GN}+o(1))^{-\frac{2}{3}}-\frac14
\theta_n^3
-\frac{1}{6}\theta_n^2(n^{-\frac{2}{3}}(\mathrm{C}_{\rm
GN}+o(1))^{-\frac{2}{3}}-\frac34\theta_n)=o(1).
\end{align}
Combining that with the non-increasing monotonicity of $c\mapsto \gamma(c)$, proved in Lemma
\ref{infV(c)}, completes the proof.
\end{proof}

\section{Proof of Theorem \ref{ground state}}\label{sec:proof main theorem}
Having all the preliminaries we are in the position to prove Theorem \ref{ground state}. We will
firstly show the existence of ground states as long as $\gamma(c)$ is strictly smaller than
$\frac{\SSS^{\frac32}}{3}$ (Lemma \ref{thm smaller than sss}). Afterwards, we show that this
assumption is guaranteed by the conditions given in Theorem \ref{ground state}. It is worth noting
that the proof of Lemma \ref{thm smaller than sss} provides with an alternative for the ones from
\cite{SoaveCritical} without any use of radial symmetry arguments, which may be of independent
interest.

\begin{lemma}\label{thm smaller than sss}
Let $(\ld_1,\ld_2)$ be unstable and let $c\in(0,\infty)$. If
$\gamma(c)\in\big(0,\frac{\SSS^{\frac32}}{3}\big)$, then \eqref{variational gamma c} has at least
an optimizer $u_c\in V_c$.
\end{lemma}
\begin{proof}
let $(u_n,\beta_n)_n$ and $(u,\beta)$ be the PS-sequence and its limit deduced from Corollary \ref{ps cor}. Using the splitting properties given by \cite[Theorem 1]{BrezisLieb1983} for $L^p$-norms and by \cite{AntonelliSparber2011} for $B(u)$ we obtain that
\begin{align*}
A(u_n-u)+A(u)&=A(u_n)+o(1),\\
B(u_n-u)+B(u)&=B(u_n)+o(1),\\
C(u_n-u)+C(u)&=C(u_n)+o(1),\\
D(u_n-u)+D(u)&=D(u_n)+o(1),
\end{align*}
where $D(v):= \|v\|_2^2$. Since $E(u)=\frac{1}{2}A(u)+\frac{1}{4}B(u)-\frac{1}{6}C(u)$, we see that
\begin{equation*}
E(u_n-u)+E(u)=E(u_n)+o(1).
\end{equation*}
From the lower semicontinuity of the $L^2$-norm we obtain that
$$\|u\|_2^2\leq \liminf_{n\to\infty}\|u_n\|_2^2=c. $$
By Lemma \ref{u not equal zero}, we know that $u\neq 0$. Consequently, we infer that $u\in V(c_1)$ for some $c_1\in(0,c]$. We also have $\gamma(c)=E(u_n)+o(1)$ from Lemma \ref{ps cor}. Hence
\begin{equation}\label{Elesso1}
E(u_n-u)+\gamma(c_1)\leq E(u_n-u)+E(u)=E(u_n)+o(1)=\gamma(c)+o(1).
\end{equation}
On the other hand, direct calculation results in
\begin{equation}\label{EminusQ}
-Q(v)+3E(v)=\frac{1}{2}A(v)+\frac{1}{2}C(v)
\end{equation}
for all $v\in H^1(\R^3)$. Using $Q(u)=0$ and $Q(u_n)=o(1)$ from Lemma \ref{ps cor} we obtain that
\begin{equation*}
Q(u_n-u)=Q(u_n-u)+Q(u)=Q(u_n)+o(1)=o(1).
\end{equation*}
Inserting this into \eqref{EminusQ}, we conclude that $E(u_n-u)\geq o(1)$, since the right-hand side of \eqref{EminusQ} is always nonnegative. From Lemma \ref{infV(c)} we know that $\gamma(c_1)\geq \gamma(c)$, therefore it follows from \eqref{Elesso1} that $E(u_n-u)\leq o(1)$. Thus $E(u_n-u)=o(1)$. Since $A(\cdot)$ and $C(\cdot)$ are nonnegative, we obtain from $E(u_n-u)= o(1)$, $Q(u_n-u)=o(1)$ and \eqref{EminusQ} that
\begin{align*}
A(u_n-u)&=o(1),\notag \\
C(u_n-u)&=o(1).
\end{align*}
But $E(u_n-u)$ is a linear combination of $A(u_n-u)$, $B(u_n-u)$ and $C(u_n-u)$, it follows immediately that $B(u_n-u)=o(1)$. Now Corollary \ref{ps cor} implies
\begin{equation*}
\frac{1}{2}A(u_n)+\beta D(u_n)+B(u_n)+C(u_n)=\frac{1}{2}A(u)+\beta D(u)+B(u)+C(u)+o(1).
\end{equation*}
Using the previous splitting properties one has then
\begin{align}
\beta D(u_n)&=\beta D(u)+o(1)\notag \\
&=\beta\big(D(u_n)-D(u_n-u)+o(1)\big).
\end{align}
From this we infer that $\beta D(u_n-u)=o(1)$. Let us now show that $\beta>0$. Due to \eqref{beta bu} it is equivalently to show that $B(u)<0$. Suppose in contrast that $B(u)\geq 0$. Since
$$0=Q(u)=A(u)+\frac{3}{4}B(u)-C(u)$$
it follows that $A(u)\leq C(u) $, or in other words $\|\nabla u\|_2^2\leq \|u\|_6^6$. Together with the Sobolev inequality $\SSS \|u\|_6^2\leq \|\nabla u\|_2^2$ we obtain that
$$\SSS^{\frac{3}{4}}\leq \|\nabla u\|_2.$$
Thus
$$E(u)=E(u)-\frac{1}{6}Q(u)=\frac{1}{3}A(u)+\frac{1}{8}B(u)\geq \frac{1}{3}A(u)\geq\frac{\SSS^{\frac{3}{2}}}{3}. $$
Since $E(u_n-u)=o(1)$ and $E(u_n)=\gamma(c)+o(1)$, from the splitting property it follows that $\gamma(c)=E(u)\geq \frac{\SSS^{\frac{3}{2}}}{3}$. But $\gamma(c)\in(0, \frac{\SSS^{\frac{3}{2}}}{3})$ according to our assumption, hence we obtain a contradiction and $B(u)<0$. Therefore $D(u_n-u)=o(1)$. Together with $A(u_n-u)=o(1)$ we infer that $u_n\rightarrow u$ in $H^1(\R^3)$ and $u\in S(c)$. Using now Lemma \ref{vc implies sc} the proof is complete.
\end{proof}

\begin{proof}[Proof of Theorem \ref{ground state}]
We begin with the proof of (iv). In fact, by the definition of a stable pair, we immediately see
that $B(u)$ is always positive for any $u\in H^1(\R^3)\setminus\{0\}$. Hence the proof of (iv)
follows the one for \cite[Thm. 1.2. 2)]{SoaveCritical} verbatim, and we omit the details here. For
(i), we note that according to \cite[Lem. 6.4]{SoaveCritical}, \eqref{threshold} can in fact be
modified to the version that there exist $C_1,C_2>0$ such that
\begin{align*}
\tilde{E}(v_{\varepsilon}^{\tilde{t}^*})\leq \frac{\SSS^{\frac{3}{2}}}{3}-C_1\varepsilon^{\frac{1}{2}}+C_2\varepsilon<\frac{\SSS^{\frac{3}{2}}}{3}
\end{align*}
for all sufficiently small $\varepsilon$. Now (i) follows from Lemma \ref{thm smaller than sss}.

(ii) is a direct consequence of Lemma \ref{vanishing of gamma c} and Lemma \ref{thm smaller than
sss} by defining
\begin{align*}
c^*:=\inf\Big\{c>0:\gamma(c)<\frac{\SSS^{\frac32}}{3}\Big\}.
\end{align*}
That the ground states are in fact solutions of \eqref{pde} with positive $\beta$ follows from Lemma \ref{vc implies sc} and the proof of Lemma \ref{thm smaller than sss}; that the solutions of \eqref{pde} are positive follows directly from the strong maximum principle. Finally, we note that in the case of (iii), if a solution of \eqref{pde} is radial, then it must be a solution of \eqref{cqnls} with some $\ld_1\geq 0$. However, this is impossible due to \cite[Thm. 1.2. 2)]{SoaveCritical}. This finishes the desired proof.
\end{proof}

\addcontentsline{toc}{section}{Acknowledgments}
\subsubsection*{Acknowledgments}
Y. Luo acknowledges the funding by Deutsche Forschungsgemeinschaft (DFG) through the Priority Programme SPP-1886.

\addcontentsline{toc}{section}{References}
\bibliographystyle{hacm}

\end{document}